\documentclass{birkjour}

\usepackage[all]{xy}
\usepackage[english]{babel}
\usepackage{amsmath}
\usepackage{amsthm}
\usepackage{amssymb}

\newcommand{\comments}[1]{} 

\newtheorem{thm}{Theorem}

\newtheorem{lem}{Lemma}
\newtheorem{cor}{Corollary}

\theoremstyle{definition}
\newtheorem{Def}{Definition}

\theoremstyle{remark}
\newtheorem{rem}{Remark}

\theoremstyle{remark}

\def\cC{\mathcal C}

\def\bR{\mathbb R} 

\def\bB{\mathbb B}

\def\bN{\mathbb N}

\def\bC{\mathbb C}

\def\Ker{\operatorname{Ker}}

\def\clifford{\cC\ell}

\def\pa{\partial}

\def\pod{\underline}
\def\nad{\overline}

\def\sl{\mathfrak{sl}}

\begin{document}

\title[Generating functions]
{Generating functions for spherical\\ harmonics and spherical monogenics}

\author{P. Cerejeiras}
\address{CIDMA - Center for Research and Development\\ in Mathematics and Applications,\\  
Department of Mathematics, University of Aveiro,\\ Campus de Santiago, P – 3810-193 Aveiro,\\ Portugal}
\email{pceres@ua.pt}

\author{U. K\"ahler}
\address{CIDMA - Center for Research and Development\\ in Mathematics and Applications,\\  
Department of Mathematics, University of Aveiro,\\ Campus de Santiago, P – 3810-193 Aveiro,\\ Portugal}
\email{ukaehler@ua.pt}

\author{R. L\' avi\v cka}
\address{Faculty of Mathematics and Physics,\\ Charles University in Prague,\\ Sokolovsk\'a 83, 186 75 Praha 8,\\ Czech Republic}
\email{lavicka@karlin.mff.cuni.cz} 



\dedicatory{To K. G\"urlebeck}

\begin{abstract} In this paper, we study generating functions for the standard orthogonal bases of spherical harmonics and spherical monogenics in $\bR^m$. 
Here spherical monogenics are polynomial solutions of the Dirac equation in $\bR^m$.
In particular, we obtain the recurrence 
formula which expresses the generating function in dimension $m$ in terms of that in dimension $m-1$. Hence we can find closed formul\ae\ of generating functions in $\bR^m$ by induction on the dimension $m$. 
\end{abstract}

\maketitle


\section{Introduction}

It is well-known that classical orthogonal polynomials can be defined by their generating functions. 
For example, the Gegenbauer polynomials $C^{\nu}_k$ are uniquely determined by the generating function
\begin{equation}\label{GF_gegen}
\frac{1}{(1-2xh+h^2)^{\nu}}=\sum_{k=0}^{\infty} C^{\nu}_k(x) h^k 
\end{equation}
where $\nu>0$, $|x|\leq 1$ and $|h|<1$ (see e.g.\ \cite[p.\ 18]{DX} or \cite[p.173]{R}).
In \cite{R}, a~general framework is developed for a~study of properties of polynomial sequences, including the Appell property and generating functions.
In this paper, we deal with generating functions for the standard orthogonal bases of spherical harmonics and spherical monogenics in $\bR^m$.

Orthogonal bases of spherical harmonics are well-known and have been studied for a long time. Spherical harmonics are useful in many theoretical areas and on applications such as structural mechanics, etc. In Clifford analysis, a similar role is played by spherical monogenics. Monogenic functions are defined as  Clifford algebra valued solutions $f$ of the equation
$\pa f=0$ where $\pa$ is the Dirac operator on $\bR^m$. 
Spherical monogenics are polynomial solutions of the Dirac equation.
Since the Dirac operator $\pa$ factorizes the Laplace operator $\Delta$ in the sense that $\Delta=-\pa^2$ 
Clifford analysis can be understood as a~refinement of harmonic analysis. 
On the other hand, monogenic functions are at the same time a~higher dimensional analogue of holomorphic functions of one complex variable. 
See \cite{BDS, DSS, GS, GHS} for an account of Clifford analysis.

The first construction of orthogonal bases of spherical monogenics valid for any dimension was given by F.~Sommen, see \cite{som,DSS}.
In dimension 3,  explicit constructions using the standard bases of spherical harmonics were done also by K. G\"urlebeck, H.~Malonek, I. Ca\c c\~ao and S.~Bock (see e.g.\ \cite{BG,cac,BockCacGue,CacGueBock,CacGueMal,CM06,CM08}). 
From the point of view of representation theory, the standard bases of spherical harmonics are nothing else than examples of the so-called Gelfand-Tsetlin bases, see \cite{mol}.
V.~Sou\v cek proposed studying these bases in Clifford analysis. 
In particular, in \cite{BGLS}, 
it is observed that the complete orthogonal system in $\bR^3$ of \cite{BG} and F. Sommen's bases \cite{som,DSS} can be both considered as Gelfand-Tsetlin bases.
Actually, it turns out that Gelfand-Tsetlin bases in all cases so far studied in Clifford analysis are, by construction, uniquely determined and orthogonal and,
in addition, they possess the so-called Appell property,  see \cite{lav_habil} for a~recent survey, \cite{lavSL2, lav_GTB_harm} for the classical Clifford analysis, 
\cite{DLS4,lav_GTB_HdR} for Hodge-de Rham systems  and \cite{GTBH,EFH} for Hermitian Clifford analysis. 
Therefore we call them the standard orthogonal bases in the sequel. For a~detailed historical account of this topic, we refer to \cite{BGLS}. 

In this paper, we study generating functions for the standard orthogonal bases of spherical harmonics and spherical monogenics in $\bR^m$. 
We obtain the recurrence 
formula which expresses the generating function in dimension $m$ in terms of that in dimension $m-1$, 
see below Theorem \ref{GF_harm} for spherical harmonics and Theorem \ref{GF_mon} for spherical monogenics. 
Using the recurrence 
formula, we can obtain closed formul\ae\ of generating functions in $\bR^m$ by induction on the dimension $m$. 
This is based on the generating function \eqref{GF_gegen} for the Gegenbauer polynomials.
It seems that analogous results can be obtained also for Hodge-de Rham systems \cite{lav_GTB_HdR} and even in Hermitian Clifford analysis \cite{EFH}.
But, in the hermitian case, the generating function for the Jacobi polynomials should be used instead of \eqref{GF_gegen}.

\section{Spherical harmonics}

In this section, we study generating functions for spherical harmonics.
Let us recall the standard construction of an orthogonal basis in 
the complex Hilbert space
$L^2(\bB_m,\bC)\cap \Ker \Delta$ of $L^2$-integrable harmonic functions $g:\bB_m\to\bC$. Here $\bB_m$ is the unit ball in $\bR^m$.
One proceeds by induction. 
Of course, the polynomials 
\begin{equation}\label{GTHarm2}
harm_{k_2}^{\pm}(x_1,x_2)=(x_1\pm ix_2)^{k_2}/({k_2}!),\ {k_2}\in\bN_0
\end{equation} 
form an orthogonal basis of the space $L^2(\bB_2,\bC)\cap \Ker \Delta$.
To construct the bases in higher dimensions, we need the embedding factors $F^{({k_m})}_{m,j}=F^{({k_m})}_{m,j}(x)$ defined as   
\begin{equation}\label{EFHarm}
F^{({k_m})}_{m,j}=|x|_m^{{k_m}}\; C^{m/2+j-1}_{{k_m}}(x_m/|x|_m),\ x\in\bR^m
\end{equation} 
where $x=(x_1,\ldots, x_m)$ and $|x|_m=\sqrt{x_1^2+\cdots+x_m^2}$.
Then, it is well-known that an orthogonal basis of the space $L^2(\bB_m,\bC)\cap \Ker \Delta$ is formed by the polynomials
\begin{equation}\label{GTHarm}
harm_{k}^{\pm}(x)=harm_{k_2}^{\pm}(x_1,x_2)\prod^m_{r=3}F^{(k_r)}_{r,k^\ast_{r-1}}
\end{equation}
where $k=(k_2, \cdots, k_m) \in \bN_0^{m-1}$ and $k^\ast_r = k_2+\cdots +k_r$. See e.g.\ \cite[p.\ 35]{DX} or \cite{lav_GTB_harm}.
In difference to \cite{lav_GTB_harm}, we use another normalization of the embedding factors $F^{({k_m})}_{m,j}$ 
and we also change the notation for indices which in turns provides a more elegant expression for generating functions. 

\begin{Def}
We define the generating function $H^{\pm}_m$ of the orthogonal basis $harm_{k}^{\pm},$ $k \in \bN_0^{m-1}$ of spherical harmonics in $\bR^m$  by
$$H^{\pm}_m(x,h)=\sum_{k \in \bN_0^{m-1}} harm_{k}^{\pm}(x)\; h^{k}$$
whenever the series on the right-hand side converges absolutely.
Here $x\in\bR^m$, $h=(h_2,\ldots,h_m)\in\bR^{m-1}$ and $h^{k}=h_2^{k_2}\cdots h_m^{k_m}$.   
\end{Def}

Obviously, the following result follows easily from \eqref{GF_gegen}.  

\begin{lem}\label{GF_F}
We have that
$$\sum_{k_m=0}^{\infty} F^{(k_m)}_{m,j}(x)\; h_m^{k_m}=
\frac{1}{(1-2x_m h_m+h^2_m|x|^2_m)^{\frac{m}{2}-1+j}}$$
where $|x|_m\leq 1$, $|h_m|<1$ and $j\in\bN_0$.
\end{lem}

Now we prove basic properties of the generating functions $H^{\pm}_m$.

\begin{thm}\label{GF_harm} For each $m\geq 2$ there is a~neighborhood $U_m$ of $0$ in $\bR^{m-1}$ such that 
the following statements hold true.
\begin{itemize}
\item[(i)] The generating functions $H^{\pm}_m(x,h)$ are defined if $|x|_m\leq 1$ and $h\in U_m$.
\item[(ii)] For each $k\in \bN_0^{m-1}$, we have that 
$$harm_{k}^{\pm}(x)=\frac{1}{k !}\;\pa^{k} H^{\pm}_m(x,h)|_{h=0},\ \ |x|_m\leq 1$$
where $k !=(k_2!)\cdots(k_m!)$ and $\pa^{k}=\pa^{k_2}_{h_2}\cdots\pa^{k_m}_{h_m}$.
\item[(iii)] For $m\geq 3$, $|x|_m\leq 1$ and $h\in U_m$, we have that 
$$H^{\pm}_m(x,h)=d_m^{1-\frac{m}{2}} H^{\pm}_{m-1}(\pod x,\pod h /d_m)$$
where $d_m= 1-2x_m h_m+h^2_m|x|^2_m,$ $\pod x=(x_1, \cdots, x_{m-1})$ and $\pod h /d_m=(h_2/d_m, \cdots, h_{m-1}/d_m).$
\end{itemize}
\end{thm}

\begin{proof}
We prove this theorem by induction on the dimension $m$. 
It is easily seen that the theorem is true for $m=2$. 
Indeed, we have that
$$H^{\pm}_2(x_1,x_2,h_2)=\sum_{k_2=0}^{\infty} \frac{(x_1\pm ix_2)^{k_2}}{k_2!} h_2^{k_2}=\exp((x_1\pm ix_2)h_2).$$
Now assume that the theorem is true for $m-1$.
Let $H^{\pm}_{m-1}(\pod x,\pod h)$ be defined for $\pod h\in U_{m-1}=(-\delta_2,\delta_2)\times\cdots\times(-\delta_{m-1},\delta_{m-1})$ and $|\pod x|_{m-1}\leq 1$ and let $|x|_m\leq 1$.
It is easy to see that
\begin{equation}\label{Hrecur}
H^{\pm}_m(x,h)=\sum_{\pod k}\left(\sum_{k_m=0}^{\infty} 
F^{(k_m)}_{m,k_{m-1}^\ast}(x)\; h_m^{k_m}\right) harm_{\pod k}^{\pm}(\pod x)\; \pod h^{\pod k}
\end{equation}
where the first sum is taken over all $\pod k=(k_2, \cdots, k_{m-1})\in \bN_0^{m-2}$.
By Lemma \ref{GF_F}, we have that
$$\sum_{k_m=0}^{\infty} F^{(k_m)}_{m,k_{m-1}^\ast}(x)\; h_m^{k_m}=
d_m^{1-\frac{m}{2}-(k_2+\cdots+k_{m-1})}$$
if $|h_m|<1$. Using this formula and \eqref{Hrecur}, we have that
$$H^{\pm}_m(x,h)=d_m^{1-\frac{m}{2}}\sum_{\pod k} harm_{\pod k}^{\pm}(\pod x)\; (\pod h/d_m)^{\pod k}
=d_m^{1-\frac{m}{2}} H^{\pm}_{m-1}(\pod x,\pod h /d_m)$$
whenever $h\in U_{m}=(-\delta_2/4,\delta_2/4)\times\cdots\times(-\delta_{m-1}/4,\delta_{m-1}/4)\times(-1/2,1/2)$.
Indeed, $d_m\geq(1-h_m|x|_m)^2>1/4$ if $|h_m|<1/2$. Hence, if $h\in U_{m}$ we have that $\pod h/d_m\in U_{m-1}$
and, by \eqref{Hrecur}, we can easily see that some rearrangement of the power series defining $H^{\pm}_m(x,h)$ converges at $h$. Then Abel's Lemma \cite[Proposition 1.5.5, p.\ 23]{KP} proves that this power series converges absolutely on the whole $U_m$, which finishes the proof of the theorem. 
\end{proof}

Using the recurrence formula (iii) of Theorem \ref{GF_harm}, we can find closed formul\ae\ of generating functions for spherical harmonics in $\bR^m$
by induction on the dimension $m$.

\begin{cor} In particular, we have the following formula 
$$H^{\pm}_3(x_1,x_2,x_3,h_2,h_3)=\frac{1}{(1-2x_3 h_3+h^2_3|x|^2_3)^{1/2}}\exp\left(\frac{(x_1\pm ix_2)h_2}{1-2x_3 h_3+h^2_3|x|^2_3}\right).$$
\end{cor}  

\begin{rem} 
It is well-known that an orthogonal basis of real valued spherical harmonics in $\bR^m$ is formed by the polynomials $\Re{harm_{k}^{+}}$, $\Im{harm_{k}^{+}}$, $k \in \bN_0^{m-1}$. Here $\Re z$ and $\Im z$ are the real and imaginary part of the complex number $z$. Hence the corresponding generating functions
are $\Re H^{+}_m$, $\Im H^{+}_m$.  
\end{rem}

\begin{rem} If one replaces in the definition of the orthogonal basis \eqref{GTHarm} the polynomials $harm_{k_2}^{\pm}(x_1,x_2)=(x_1\pm ix_2)^{k_2}/(k_2!)$ with 
\begin{equation}\label{GTHarm2nad}
\nad{harm}_{k_2}^{\;\pm}(x_1,x_2)=(x_1\pm ix_2)^{k_2}, 
\end{equation} 
the corresponding generating functions $\nad H^{\;\pm}_m$ are definitely different from $H^{\;\pm}_m$ but they obviously satisfy again Theorem \ref{GF_harm}.
In particular, we have that
$$\nad H^{\;\pm}_2(x_1,x_2,h_2)=\sum_{k_2=0}^{\infty} (x_1\pm ix_2)^{k_2} h_2^{k_2}=\frac{1-(x_1\mp x_2i)h_2}{1-2x_1 h_2+h^2_2|x|^2_2}.$$
\end{rem}

\section{Spherical monogenics} 

In this section, we introduce and investigate generating functions for spherical monogenics.
For an account of Clifford analysis, we refer to \cite{BDS, DSS, GS, GHS}.
Denote by $\clifford_m$ either the real Clifford algebra $\bR_{0,m}$ or the complex one $\bC_m$,
generated by the vectors $e_1,\ldots,e_m$ such that
$e_j^2=-1$ for $j=1,\ldots,m.$
As usual, we identify a~vector $x=(x_1,\ldots,x_m)\in\bR^m$ with the element $x_1e_1+\cdots+x_me_m$ of the Clifford algebra $\clifford_m$.
Let $G\subset\bR^m$ be open. Then a~continuously differentiable function $f:G\to\clifford_m$ is called monogenic if it satisfies the equation $\pa f=0$ on $G$ where
the Dirac operator $\pa$ is defined as
\begin{equation}\label{Dirac}
\pa=e_1\pa_{x_1}+\cdots+e_m\pa_{x_m}.
\end{equation}

Denote by $L^2(\bB_m,\clifford_m)\cap \Ker \pa$ the space of $L^2$-integrable monogenic functions $g:\bB_m\to\clifford_m$. 
It is well-known that $L^2(\bB_m,\clifford_m)\cap \Ker \pa$ forms the right $\clifford_m$-linear Hilbert space.
Let us  recall a~construction of an orthogonal basis in this space, see \cite{lav_GTB_harm} for more details.
It is easy to see that the polynomials
\begin{equation}\label{GTMon2}
mon_{k_2}(x_1,x_2)=(x_1-e_{12} x_2)^{k_2}/({k_2}!),\ {k_2}\in\bN_0
\end{equation} 
form an orthogonal basis of the space $L^2(\bB_2,\clifford_2)\cap \Ker \pa$.
Here we write $e_{12}=e_1e_2$ as usual. 
To construct the bases in higher dimensions, we need the embedding factors $X^{({k_m})}_{m,j}=X^{({k_m})}_{m,j}(x)$ defined as   
\begin{equation}\label{EFMon}
X^{({k_m})}_{m,j}=\frac{m-2+{k_m}+2j}{m-2+2j}\;F^{({k_m})}_{m,j}(x)+
F^{({k_m}-1)}_{m,j+1}(x)\;\pod x e_m,\ x\in\bR^m
\end{equation} 
where $\pod x=x_1e_1+\cdots+x_{m-1} e_{m-1}$, $F^{({k_m})}_{m,j}$ are given in \eqref{EFHarm} and $F^{(-1)}_{m,j+1}=0$.
Then it is well-known that an orthogonal basis of the space $L^2(\bB_m,\clifford_m)\cap \Ker \pa$ is formed by the polynomials
\begin{equation}\label{GTMon}
mon_{k}(x)=X^{(k_m)}_{m,k_{m-1}^\ast}X^{(k_{m-1})}_{m-1,k_{m-2}^\ast}\cdots X^{(k_3)}_{3,k_2^\ast}\; mon_{k_2}(x_1,x_2)
\end{equation}
where $k=(k_2, \cdots, k_m) \in \bN_0^{m-1}$ and $k^\ast_r = k_2+\cdots +k_r$.
Let us remark that due to non-commutativity of the Clifford multiplication the order of factors in the product \eqref{GTMon} is important.
See \cite{lav_GTB_harm} for more details.
In comparison with \cite{lav_GTB_harm}, we use another normalization of the embedding factors $X^{({k_m})}_{m,j}$ 
and we also change the notation for indices to get a~nice expression for generating functions. 

\begin{Def}
We define the generating function $M_m$ of the orthogonal basis $mon_{k},$ $k \in \bN_0^{m-1}$ of spherical monogenics in $\bR^m$  by
$$M_m(x,h)=\sum_{k \in \bN_0^{m-1}} mon_{k}(x)\; h^{k}$$
whenever the series on the right-hand side converges absolutely.
Here $x\in\bR^m$ and $h=(h_2,\ldots,h_m)\in\bR^{m-1}$.   
\end{Def}

In particular, it is easily seen that 
$$M_2(x_1,x_2,h_2)=\sum_{k_2=0}^{\infty} \frac{(x_1-e_{12} x_2)^{k_2}}{k_2!} h_2^{k_2}=\exp((x_1-e_{12} x_2)h_2).$$
Here $\exp((x_1-e_{12} x_2)h_2)=\exp(x_1h_2)(\cos(x_2h_2)-e_{12}\sin(x_2h_2))$.
To study the generating functions in higher dimensions we need to know the generating function of the embedding factors $X^{(k_m)}_{m,j}$.

\begin{lem}\label{GF_X}
We have that
$$\sum_{k_m=0}^{\infty} X^{(k_m)}_{m,j}(x) h_m^{k_m}=\frac{1+xh_me_m}{(1-2x_m h_m+h^2_m|x|^2_m)^{m/2+j}}$$
where $|x|_m\leq 1$, $|h_m|<1$ and $j\in\bN_0$.
\end{lem}

\begin{proof}
Put $\nu=m/2-1+j$.
By \eqref{EFMon}, the series we want to sum up is equal to
$$\sum_{k_m=0}^{\infty} \frac{k_m+2\nu}{2\nu}\;F^{({k_m})}_{m,j}(x) h_m^{k_m}+
\sum_{k_m=1}^{\infty} F^{({k_m}-1)}_{m,j+1}(x)h_m^{k_m}\;\pod x e_m=\Sigma_1+\Sigma_2.$$
Obviously, by Lemma \ref{GF_F}, we get that
$$\Sigma_2=\frac{\pod x h_me_m}{(1-2x_m h_m+h^2_m|x|^2_m)^{\nu+1}}.$$
Moreover, using Lemma \ref{GF_F} again, we have that
$$\Sigma_1=\frac{h_m}{2\nu}\;\sum_{k_m=1}^{\infty} F^{({k_m})}_{m,j}(x) k_mh_m^{k_m-1}+\frac{1}{(1-2x_m h_m+h^2_m|x|^2_m)^{\nu}}$$
and hence
$$\Sigma_1=\frac{h_m}{2\nu}\;\frac{d\ }{dh_m} \frac{1}{(1-2x_m h_m+h^2_m|x|^2_m)^{\nu}}+\frac{1}{(1-2x_m h_m+h^2_m|x|^2_m)^{\nu}},$$
which gives
$$\Sigma_1=\frac{1-x_mh_m}{(1-2x_m h_m+h^2_m|x|^2_m)^{\nu+1}}.$$
Finally, we conclude that
$$\Sigma_1+\Sigma_2=\frac{1+xh_me_m}{(1-2x_m h_m+h^2_m|x|^2_m)^{m/2+j}},$$
which finishes the proof.
\end{proof}

Now we can prove basic properties of the generating functions $M_m$ quite similarly as in the harmonic case if, in this case, we use Lemma \ref{GF_X} instead of Lemma~\ref{GF_F}. Then we obtain the following result.

\begin{thm}\label{GF_mon} For each $m\geq 2$ there is a~neighborhood $U_m$ of $0$ in $\bR^{m-1}$ such that 
the following statements hold true.
\begin{itemize}
\item[(i)] The generating functions $M_m(x,h)$ are defined if $|x|_m\leq 1$ and $h\in U_m$.
\item[(ii)] For each $k\in \bN_0^{m-1}$, we have that 
$$mon_{k}(x)=\frac{1}{k !}\;\pa^{k} M_m(x,h)|_{h=0},\ \ |x|_m\leq 1$$
where $k !=(k_2!)\cdots(k_m!)$ and $\pa^{k}=\pa^{k_2}_{h_2}\cdots\pa^{k_m}_{h_m}$.
\item[(iii)] For $m\geq 3$, $|x|_m\leq 1$ and $h\in U_m$, we have that 
$$M_m(x,h)= (1+xh_me_m)\; d_m^{-\frac{m}{2}} M_{m-1}(\pod x,\pod h /d_m)$$
where $d_m= 1-2x_m h_m+h^2_m|x|^2_m,$ $\pod x=(x_1, \cdots, x_{m-1})$ and $\pod h /d_m=(h_2/d_m, \cdots, h_{m-1}/d_m).$
\end{itemize}
\end{thm}

Using the recurrence formula (iii) of Theorem \ref{GF_mon}, we can find closed formul\ae\ of generating functions for spherical monogenics in $\bR^m$
by induction on the dimension $m$.

\begin{cor} In particular, we have the following formula 
$$M_3(x_1,x_2,x_3,h_2,h_3)=\frac{1+xh_3e_3}{(1-2x_3 h_3+h^2_3|x|^2_3)^{3/2}}\exp\left(\frac{(x_1-e_{12}x_2)h_2}{1-2x_3 h_3+h^2_3|x|^2_3}\right).$$
\end{cor}

\begin{rem} If one replaces in the definition of the orthogonal basis \eqref{GTMon} the polynomials 
$mon_{k_2}(x_1,x_2)=(x_1-e_{12} x_2)^{k_2}/({k_2}!)$ with 
\begin{equation}\label{GTMon2nad}
\nad{mon}_{k_2}(x_1,x_2)=(x_1-e_{12} x_2)^{k_2}, 
\end{equation} 
the corresponding generating functions $\nad M_m$ are different from $M_m$ but they obviously satisfy again Theorem \ref{GF_mon}.
In particular, we have that
$$\nad M_2(x_1,x_2,h_2)=\sum_{k_2=0}^{\infty} (x_1-e_{12} x_2)^{k_2} h_2^{k_2}=\frac{1-(x_1+e_{12} x_2)h_2}{1-2x_1 h_2+h^2_2|x|^2_2}.$$
\end{rem}

\subsection*{Acknowledgements} 
We would like to thank V. Sou\v cek for useful discussions on this topic.
The work of the first and second authors was supported by Portuguese funds through the CIDMA - Center for Research and Development in Mathematics and Applications, and the Portuguese Foundation for Science and Technology(``FCT - Funda\c c\~ao para a Ci\^encia e a Tecnologia''), within project PEst-OE/MAT/UI4106/2014.


\end{document}